\documentclass[english,11pt]{amsart}
\usepackage{amsopn,amsthm,amsfonts,amsmath,amssymb,amscd}
\usepackage{babel}
\usepackage{nicematrix}
\usepackage{mathtools}
\usepackage{amssymb,tikz}

\newtheorem{Theorem}{Theorem}[section]
\newtheorem{Lemma}[Theorem]{Lemma}

\newtheorem{Corollary}[Theorem]{Corollary}
\newtheorem{Example}[Theorem]{Example}
\newtheorem*{Remark}{Remark}

\newenvironment{Proof*}{{\it Proof.}}

\newcommand{\RR}{\mathbb{R}}
\newcommand{\ZZ}{\mathbb{Z}}

\newcommand{\Tr}{\mathrm{tr}}

\newcommand{\OO}{\mathcal{O}}

\begin{document}

\title{Products of nilpotents in a quaternion ring of odd order}

\author{David Dol\v{z}an}
\address{Department of Mathematics, Faculty of Mathematics and Physics, University of Ljubljana, Jadranska 19, SI-1000 Ljubljana, Slovenia; and Institute of Mathematics, Physics and Mechanics, Jadranska 19, SI-1000 Ljubljana, Slovenia}
\email{david.dolzan@fmf.uni-lj.si}

\thanks{The author acknowledges the financial support from the Slovenian Research Agency  (research core funding No. P1-0222)}

\begin{abstract} 
Let $R$ be a finite commutative local principal ring of cardinality $q^n$, where
$q = p^r$ for an odd prime $p$ and integer $r$ with $R/J(R) \simeq GF(q)$. We determine the number of elements in the quaternion ring $H(R)$ that can be expressed as a product of at least 
$2n-1$ nilpotent elements, and show by example that this bound is sharp.
\end{abstract}

\maketitle 

 \section{Introduction}

\bigskip

Let $R$ be a commutative ring. The set
\begin{equation*}
H(R)= \{r_1+ r_2i + r_3j + r_4k : r_i \in R\}=R \oplus Ri \oplus Rj \oplus Rk,
\end{equation*}
together with the relations $ i^2 = j^2 = k^2 = ijk = -1$, and $ij = -ji$ turns out to be a ring, called the $\it{quaternion}$ $\it{ring}$ over $R$. This is obviously a generalization of Hamilton's division ring of real quaternions, $H(\RR)$. Recently, many studies have been devoted to quaternion rings and their properties. Among others, the authors in \cite{aris1, aris2} studied the ring $H(\ZZ_p)$, for a prime number $p$, while in \cite{mig2,mig1} the structure of $H(\ZZ_n)$ was studied.
In \cite{ghara}, the authors investigated the structure of superderivations of quaternion rings over $\ZZ_2$-graded rings, while in \cite{ghara2} the authors studied the mappings on quaternion rings.
In \cite{cher22}, the structure of the ring $H(R)$ was described for a ring $R$, where $2$ is an invertible element, the authors in \cite{xie} studied systems of matrix equations over $H(R)$, and in \cite{cherdol}, the representations of elements of a quaternion ring as sums of exceptional units were investigated.

In any unital ring, the set of units, the set of nilpotents, and the set of idempotents are of utmost importance. By taking sums of elements in these sets, clean elements (sums
of units and idempotents), nil-clean elements (sums of nilpotents and idempotents)
and fine elements (sums of units and nilpotents) were defined (clean in \cite{nicholson}, nil-clean
in \cite{diesl} and fine in \cite{caluga0}) and accordingly clean (or nil-clean or fine) rings were defined and studied in depth.

Recently, in \cite{caluga}, the author has studied a problem of decomposing elements as a product of two idempotents in the ring of $2$-by-$2$ matrices over a general domain, and later in \cite{caluga1} the author also examined the question of which $2$-by-$2$ idempotent matrices are products of two nilpotent matrices. Furthermore, the characterization of singular $2$-by-$2$ matrices over a commutative domain which are products of two idempotents or products of two nilpotents was investigated in \cite{caluga2}.
In \cite{dolzanid} the author investigated the product of idempotents in quaternion rings.



In this paper, we investigate which elements of a quaternion ring over a finite commutative local principal ring can be written as a product of nilpotents.
The paper is structured as follows. In the next short section, we gather all the necessary definitions and basic lemmas that we shall use throughout the paper. Section 3 contains the main results, where we investigate the quaternion ring $H(R)$ over a finite principal commutative ring $R$ of odd order $q^n$, where $q = p^r$ for an odd prime $p$ and integer $r$ with $R/J(R) \simeq GF(q)$. We determine the number of elements that can be decomposed as a product of at least $2n-1$ nilpotents (see Theorem \ref{main}). We do this by considering the isomorphic structure of $2$-by-$2$ matrices over the same ring, and we characterize all the $2$-by-$2$ matrices that can be decomposed as a product of at least $2n-1$ nilpotents (see Theorem \ref{product} and Corollary \ref{charact}). We also provide an example (see Example \ref{counter}) that this result fails if fewer than $2n-1$ nilpotents are used in the decomposition.

\bigskip

 \section{Definitions and preliminaries}
\bigskip

All rings in our paper will be finite rings with identity. 
For a ring $R$, $Z(R)$ will denote its centre and $N(R)$ will denote the set of all its nilpotents. The group of units in $R$ will be denoted by $U(R)$ and the Jacobson radical of $R$ by $J(R)$. 

We will denote the $2$-by-$2$ matrix ring with entries in a subring $S$ of the ring $R$ by $M_2(S)$, while the group of invertible matrices therein will be denoted by $GL_2(S)$.


For any $a, b \in R$, let $M(a,b)=\left(\begin{array}{cc}
a & b\\
0 & 0
\end{array}\right) \in M_2(R)$. The trace of a matrix $A \in M_2(R)$ will be denoted by $\Tr(A)$. 

The group $GL_2(R)$ acts on $M_2(R)$ by conjugation. We will denote the orbit of an element $X \in M_2(R)$ for this action by $\OO_X$.

It follows from \cite[Theorem 2]{ragha} that every finite local ring has cardinality $p^{nr}$ for some prime number $p$ and some integers $n, r$. Furthermore, the Jacobson radical $J(R)$ is of cardinality $p^{(n-1)r}$ and the factor ring $R/J(R)$ is a field with $p^r$ elements (denoted $GF(p^r)$).

We also have the following lemma, which can be found in \cite{dolzanid}.

\begin{Lemma} \cite[Lemma 2.2]{dolzanid}
\label{pid}    
Let $R$ be a finite local principal ring of cardinality $q^{n}$, where $q=p^r$ for some prime number $p$ and integers $n, r$ with $R/J(R) \simeq GF(q)$. Then there exists $x \in J(R)$ such that $J(R)^k=(x^k)$ for every $k \in \{0,1,\ldots,n\}$. In particular, $|J(R)^k|=q^{n-k}$ for every $k \in \{0,1,\ldots,n\}$.
\end{Lemma}

The following lemma will be used frequently throughout the paper. The proof is straightforward.

\begin{Lemma}
\label{obvious}
   Let $R$ be a finite ring. Then $U(R)+J(R) \subseteq U(R)$. 
\end{Lemma}

\bigskip

 \section{Products of nilpotents}
\bigskip

In this section we turn to the product of nilpotents in a quaternion ring over a finite commutative local ring. 
The following theorem shows that if $R$ is a finite commutative local ring of odd order, then we can study $2$-by-$2$ matrices instead of $H(R)$.

\begin{Theorem}\cite[Theorem 3.10]{cher}
\label{izo}
If $R$ is a finite commutative local ring of odd cardinality $p^{nr}$ such that $R/J(R)$ is a field with $q=p^r$ elements, then $H(R) \simeq M_2(R)$.
\end{Theorem}

Next, we find the number of nilpotent elements in $H(R)$.

\begin{Lemma}
\label{noofnil}
Let $R$ be a finite commutative local ring of
odd cardinality $p^{nr}$ such that $R/J(R)$ is a field with $q=p^r$ elements.
Then $|N(H(R))|=q^{2(2n-1)}$.
\end{Lemma}
\begin{proof}
 By Theorem \ref{izo}, we have $H(R) \simeq M_2(R)$, so $|N(H(R))|=|N(M_2(R))|$.
 Since $|R/J(R)|=q$, we have $|J(R)|=q^{n-1}$, so $|J(M_2(R))|=|M_2(J(R))|=q^{4(n-1)}$.
Also, $M_2(R)/J(M_2(R)) \simeq M_2(GF(q))$, so \cite[Theorem 3.2]{cher} now gives us
$|N(M_2(R))|=q^{4(n-1)}q^2=q^{2(2n-1)}$.
\end{proof}

So, let us study the nilpotents in a $2$-by-$2$ matrix ring.

\begin{Lemma}
\label{matnil}
Let $R$ be a finite commutative local ring.
Then $A \in M_2(R)$ is nilpotent if and only if one of the following statements hold.
\begin{enumerate}
    \item 
    $A \in M_2(J(R)),$
    \item 
    $A \in \left(\begin{array}{cc}
0 & u\\
0 & 0
\end{array}\right) + M_2(J(R))$ for some $u \in U(R)$,
    \item 
    $A \in \left(\begin{array}{cc}
0 & 0\\
u & 0
\end{array}\right) + M_2(J(R))$ for some $u \in U(R)$,
\item
    $A \in \left(\begin{array}{cc}
u & -v\\
v^{-1}u^2 & -u
\end{array}\right) + M_2(J(R))$ for some $u, v \in U(R)$.
\end{enumerate}
\end{Lemma}
\begin{proof}
Choose $A$ in any of the four sets above.
Since $J(R)$ is nilpotent, we know that $M_2(J(R))$ is nil, and $A^2 \in M_2(J(R))$, so we conclude that $A$ is nilpotent.
On the other hand, suppose that $A \in M_2(R)$ is nilpotent. Denote $F=R/J(R)$ and let  $\overline{A}=A+J(M_2(R)) \in M_2(R)/J(M_2(R)) \simeq M_2(F)$ for any $A \in M_2(R)$.
 Then $\overline{A} \in M_2(F)$ is a nilpotent matrix of rank at most $1$, and thus $\overline{A}=xy^T$ for some $x,y \in F^2$ with $x^Ty=0$. Denote $(x_1,x_2)=x^T$ and $(y_1,y_2)=y^T$. Suppose $\overline{A}\neq 0$. If $x_1=0$, then $y_2=0$ and $\overline{A} = \left(\begin{array}{cc}
0 & 0\\
x_2y_1 & 0
\end{array}\right)$. If $x_2=0$, then $y_1=0$ and $\overline{A} = \left(\begin{array}{cc}
0 & x_1y_2\\
0 & 0
\end{array}\right)$. However, if $x_1, x_2 \neq 0$, then $y_1, y_2 \neq 0$. Denote $u=x_1y_1$ and $v=-x_1y_2$. Then the fact that $x_1y_1+x_2y_2=0$ yields $x_2y_2=-u$ and $x_2y_1=-x_1^{-1}x_2^2y_2=v^{-1}u^2$. 
\end{proof}

The next technical lemma will be crucial in our investigation.

\begin{Lemma}
\label{technical}
    Let $R$ be a finite commutative local principal ring of cardinality $q^{n}$, where $q=p^r$ for some prime number $p$ and integers $n, r$ with $R/J(R) \simeq GF(q)$. Then the following two statements hold.
\begin{enumerate}
    \item 
        Suppose $A=\left(\begin{array}{cc}
t & j_1\\
j_2 & 0
\end{array}\right) \in M_2(R)$ for some $t \in J(R)^k\setminus J(R)^{k+1}$, $j_1 \in J(R)^l\setminus J(R)^{l+1}$ with $0 \leq k < l \leq n$ and $j_2 \in J(R)^{n-1}$. Then there exist $a, b \in R$ such that $A \in \OO_{M(a,b)}$.
  \item 
          If $A=\left(\begin{array}{cc}
0 & j_1 \\
0 & j_2 
\end{array}\right) \in M_2(R)$ for some $j_1, j_2 \in J(R)^{n-1}$ then there exist $a, b \in R$ such that $A \in \OO_{M(a,b)}$.
\end{enumerate}
\end{Lemma}

\begin{Remark}
    By Lemma \ref{pid}, we know that $J(R)^n=0$. By a slight abuse of notation, we shall abide by the convention that $J(R)^n\setminus J(R)^{n+1}$ equals $\{0\}$, since it simplifies many of our statements. We shall use this same convention throughout the remainder of this paper.
\end{Remark}

\begin{proof}
Notice firstly that by Lemma \ref{pid}, there exists $x \in R$ such that we have $J(R)=(x)$.
\begin{enumerate}
    \item 
    Let $A=\left(\begin{array}{cc}
t & j_1\\
j_2 & 0
\end{array}\right) \in M_2(R)$. We have $A \in \OO_{M(a,b)}$ for some $a, b \in R$ if and only if there exists $P=\left(\begin{array}{cc}
\alpha & \beta \\
\gamma & \delta
\end{array}\right) \in GL_2(R)$ such that $M(a,b)P=PA$. This means that we get the following four equations:
\begin{align*}
  \alpha t + \beta j_2 = \alpha a + \gamma b, \, \alpha j_1=\beta a + \delta b, \\  \gamma t + \delta j_2 = 0, \, \gamma j_1 = 0.  
\end{align*}
We can assume that $j_2=ux^{n-1}$ for some $u \in U(R)$, since we have $A = {M(t,j_1)}$ if $j_2=0$. By our assumptions, we also have $j_1=vx^l$ and $t=wx^k$ for some $v,w \in U(R)$. 
Choose $\alpha=1$, $\beta=0$, $\gamma=x^{n-1-k}$ and $\delta=-u^{-1}w$ and observe that $P$ is indeed an invertible matrix. Then choose $b=-uw^{-1}j_1$ and $a=t$. Observe that $\delta b = j_1$ and since $k < l$, we also have $\gamma b = \gamma j_1 = 0$. Thus, one can easily check that all the four above equations do indeed hold.
  \item 
  Let $A=\left(\begin{array}{cc}
0 & j_1\\
0 & j_2
\end{array}\right)$. If $j_2=0$, then $A=M(j_1,0)$, and if $j_1=0$ then $A \in \OO_{M(j_2,0)}$.
So, we can assume that $j_1=u_1x^{n-1}$ and $j_2=u_2x^{n-1}$ for some $u_1,u_2 \in U(R)$. But then one can check that
$\left(\begin{array}{cc}
1 & 0\\
-u_2 & u_1
\end{array}\right)A\left(\begin{array}{cc}

1 & 0\\
-u_2 & u_1
\end{array}\right)^{-1}=M(j_2,x^{n-1})$.
\end{enumerate}
\end{proof}

The next lemma will investigate conjugating matrices with some special invertible matrices.

\begin{Lemma}
\label{conjug}
   Let $R$ be a commutative ring. For any $t \in R$ and any $\alpha \in U(R)$, denote $T_t=\left(\begin{array}{cc}
1 & t\\
0 & 1
\end{array}\right) \in GL_2(R)$ and $V_\alpha=\left(\begin{array}{cc}
1 & 0\\
0 & \alpha
\end{array}\right) \in GL_2(R)$. Then the following statements hold.
\begin{enumerate}
    \item 
    For any $a, b \in R$ and any $t \in U(R)$, we have $T_t^{-1}M(a,b)T_t=M(a,b+at)$.

    \item 
    For any $u, v \in U(R)$, we have $T_{vu^{-1}}^{-1}\left(\begin{array}{cc}
u & -v\\
v^{-1}u^2 & -u
\end{array}\right)T_{vu^{-1}}=\left(\begin{array}{cc}
0 & 0\\
v^{-1}u^2 & 0
\end{array}\right)$.
\end{enumerate}
\end{Lemma}
\begin{proof}
    Both statements can be checked with simple direct calculations.
\end{proof}

The next two technical lemmas will be useful.

\begin{Lemma}
\label{mab}    
   Let $R$ be a commutative ring and let $a,b \in R$. If $A \in \OO_{M(a,b)}$ and $B \in M_2(R)$ then there exist $c,d \in R$ such that $AB \in \OO_{M(c,d)}$.
\end{Lemma}
\begin{proof}
    Let $A=P^{-1}M(a,b)P$ for some invertible matrix $P \in M_2(R)$. Observe that $AB=P^{-1}M(a,b)PBP^{-1}P$ and denote $C=PBP^{-1}=\left(\begin{array}{cc}
c_1 & c_2\\
c_3 & c_4
\end{array}\right)$. Then $AB=P^{-1}M(a,b)CP=P^{-1}M(ac_1+bc_3,ac_2+bc_4)P$, so $AB \in \OO_{M(ac_1+bc_3,ac_2+bc_4)}$.
\end{proof}

\begin{Lemma}
\label{deter}
    Let $R$ be a finite commutative local principal ring of cardinality $q^{n}$, where $q=p^r$ for some prime number $p$ and integer $r$ with $R/J(R) \simeq GF(q)$. Choose $a \in J(R)^l \setminus J(R)^{l+1}$ and $b \in J(R)^k \setminus J(R)^{k+1}$ for some $n-2 \geq k > l$. If $M(a,b)+T=N_1N_2\ldots N_{2n-3}$ where $N_i \in M_2(R)$ is a nilpotent matrix for every $1 \leq i \leq 2n-3$ and $T=\left(\begin{array}{cc}
t_1 & t_2\\
t_3 & t_4
\end{array}\right) \in M_2(J(R)^{n-1})$, then $t_4=0$.
\end{Lemma}
\begin{proof}
  Let $S$ be a finite local principal ring of cardinality $q^{2n}$ with $S/J(S) \simeq GF(q)$ such that $S/J(S)^n \simeq R$. Now lift the equation $N_1N_2\ldots N_{2n-3}= M(a,b)+T$ to $S$: $\hat{N_1}\hat{N_2}\ldots \hat{N}_{2n-3}= M(\hat{a},\hat{b})+\hat{T}$.
  Since $\hat{N_i}$ is not invertible, we have $\det\hat{N_i} \in J(S)$ for every $1 \leq i \leq 2n-3$, so $\det(\hat{N_1}\hat{N_2}\ldots \hat{N}_{2n-3}) \in J(S)^{2n-3}$. Denote $\hat{T}=\left(\begin{array}{cc}
\hat{t_1} & \hat{t_2}\\
\hat{t_3} & \hat{t_4}
\end{array}\right)$. On the other hand, $\det({M(\hat{a},\hat{b})}+\hat{T})=(\hat{a}+\hat{t_1})\hat{t_4} - (\hat{b}+\hat{t_2})\hat{t_3}$. If $t_4\neq 0$, then $\hat{t_4} \in J(S)^{n-1} \setminus J(S)^{n}$ and therefore
$\hat{a}\hat{t_4} \in J(S)^{l+n-1} \setminus J(S)^{l+n}$. Note however that $\hat{t_1}\hat{t_4}, \hat{t_2}\hat{t_3} \in J(S)^{2n-2} \subseteq J(S)^{l+n}$ and $\hat{b}\hat{t_3} \in J(S)^{n-1+k} \subseteq J(S)^{l+n}$. Therefore, $(\hat{a}+\hat{t_1})\hat{t_4} - (\hat{b}+\hat{t_2})\hat{t_3} \notin J(S)^{l+n}$, a contradiction.
\end{proof}

The next theorem is extremely important for establishing our main result, as it shows that any product of a large enough number of nilpotent matrices lies in the orbit of some matrix $M(a,b)$.

\begin{Theorem}
\label{product}
    Let $R$ be a finite commutative local principal ring of cardinality $q^{n}$, where $q=p^r$ for some prime number $p$ and integers $n, r$ with $R/J(R) \simeq GF(q)$. Choose $s \geq 2n-1$ and suppose that $A \in M_2(R)$ is a product of $s$ nilpotent matrices. 
    Then there exist $a, b \in R$ such that $A \in \OO_{M(a,b)}$.
\end{Theorem}
\begin{proof}
Observe firstly that $J(R)=(x)$ and $J(R)^n=0$ by Lemma \ref{pid}.

We will prove this statement by induction on $n$. Let us start with the case $n=1$.
If $A \in M_2(GF(q))$ is a nilpotent matrix, we have the following four possibilities by Lemma \ref{matnil}:
\begin{enumerate}
    \item 
    If $A=0$, then $A \in \OO_{M(0,0)}$.
    \item
        If $A = \left(\begin{array}{cc}
0 & u\\
0 & 0
\end{array}\right)$ for some $u \in U(GF(q))$, then $A \in \OO_{M(0,u)}$.
    \item 
    If $A = \left(\begin{array}{cc}
0 & 0\\
u & 0
\end{array}\right)$ for some $u \in U(R)$, then $\left(\begin{array}{cc}
0 & 1\\
1 & 0
\end{array}\right)^{-1}A\left(\begin{array}{cc}
0 & 1\\
1 & 0
\end{array}\right)=M(0,u)$, so $A \in \OO_{M(0,u)}$.
\item
    Finally, if $A = \left(\begin{array}{cc}
u & -v\\
v^{-1}u^2 & -u
\end{array}\right)$ for some $u, v \in U(GF(q))$, then by Lemma \ref{conjug}(2) and by the previous case, $A \in \OO_{M(0,v^{-1}u^2)}$.
\end{enumerate}
So, if $A_1A_2\ldots A_s$ is a product of nilpotent matrices for some $s \geq 1$, we have $A_1 \in \OO_{M(a,b)}$ for some $a, b \in R$, therefore $A_1A_2\ldots A_s \in \OO_{M(c,d)}$ for some $c, d \in R$ by Lemma \ref{mab}. Thus, we have proved the statement holds for $n=1$.

Now, assume that $n \geq 2$ and that the statement holds for all integers $1,2,\ldots,n-1$. Let $A=N_1N_2\ldots N_s$ for some $s \geq 2n-1$ and for some nilpotent matrices $N_1,N_2,\ldots,N_s \in M_2(R)$. 
Denote $R_1=R/J(R)^{n-1}$ and observe that $R_1$ is a finite local principal ring of cardinality $q^{n-1}$  with $R_1/J(R_1) \simeq GF(q)$.
Denote also  $\overline{A}=A+J(M_2(R))^{n-1} \in M_2(R)/J(M_2(R))^{n-1}$ for any $A \in M_2(R)$. 

By induction, we have $\overline{N_1} \cdot \overline{N_2} \cdot \ldots  \cdot \overline{N_{2n-3}} \in \OO_{M(\overline{a},\overline{b})}$ for some $\overline{a},\overline{b} \in R_1$.  Therefore, there exist $a,b \in R$ such that $N_1N_2\ldots N_{2n-3} + T \in \OO_{M(a,b)}$ for some matrix $T \in M_2(J(R)^{n-1})$. It follows that there exists an invertible matrix $P \in M_2(R)$ such that  $P^{-1}(N_1N_2\ldots N_{2n-3} + T)P=(P^{-1}N_1P)(P^{-1}N_2P)\ldots (P^{-1}N_{2n-3}P) + P^{-1}TP=M(a,b)$. Suppose that $a \in J(R)^l \setminus J(R)^{l+1}$ and that $b \in J(R)^k \setminus J(R)^{k+1}$ for some integers $0 \leq k, l \leq n-2$. By Lemma \ref{conjug}(1), we can assume that $k \leq l$, since we can replace $b$ by $b+ta$ for any $t \in R$.
It immediately follows that
\begin{multline}
 \label{produkt}   
(P^{-1}N_1P)(P^{-1}N_2P)\ldots (P^{-1}N_{2n-3}P)(P^{-1}N_{2n-2}P)(P^{-1}N_{2n-1}P)=\\(M(a,b)+T')(P^{-1}N_{2n-2}P)(P^{-1}N_{2n-1}P)
\end{multline}
for $T'=-P^{-1}TP \in M_2(J(R)^{n-1})$.
Denote $N'=(P^{-1}N_{2n-2}P)$ and $N''=(P^{-1}N_{2n-1}P)$. By Lemma \ref{matnil}, we have four possibilities for the nilpotent matrix $N'$:
\begin{enumerate}
    \item 
    $N' \in M_2(J(R))$. Then $(M(a,b)+T')N'=M(a,b)N'=M(c,d)$ for some $c,d \in R$.

\item 
$N'=\left(\begin{array}{cc}
0 & 0\\
u & 0
\end{array}\right) + X'$, for some $u \in U(R)$ and $X'=\left(\begin{array}{cc}
x_1 & x_2\\
x_3 & x_4
\end{array}\right) \in M_2(J(R))$. Denote $T'=\left(\begin{array}{cc}
t_1 & t_2\\
t_3 & t_4
\end{array}\right)$. Then 
$$(M(a,b)+T')N'=\left(\begin{array}{cc}
ax_1+b(u+x_3)+t_2u & ax_2+bx_4\\
t_4u & 0
\end{array}\right).$$
We have $t_4u \in J(R)^{n-1}$. The fact that $u \in U(R)$ coupled with the fact that $k \leq l$ gives us $ax_1+b(u+x_3)+t_2u \in J(R)^k \setminus J(R)^{k+1}$, while 
$ax_2+bx_4 \in J(R)^{k+1}$.
By Lemma \ref{technical}(1), we get that $(P^{-1}N_1P)(P^{-1}N_2P)\ldots (P^{-1}N_{2n-3}P)N' \in \OO_{M(c,d)}$ for some $c,d \in R$, and therefore $(P^{-1}N_1P)(P^{-1}N_2P)\ldots (P^{-1}N_{k}P) \in \OO_{M(c',d')}$ for some $c',d' \in R$ by Lemma \ref{mab}. Thus, $N_1N_2\ldots N_k \in \OO_{M(c',d')}$ as well.

\item
$N'=
\left(\begin{array}{cc}
u & -v\\
v^{-1}u^2 & -u
\end{array}\right) 
 + X'$, for some $u, v \in U(R)$ and some $X' \in M_2(J(R))$. Let $a=u_ax^l$ and $b=u_bx^k$ for some $u_a,u_b \in U(R)$. We consider the following cases.
 \begin{itemize}
  \item
  $k < l$ or $uu_b+vu_a \notin J(R)$: Lemma \ref{conjug}(2) yields the existence of an invertible matrix $T_{vu^{-1}} \in M_2(R)$ such that $T_{vu^{-1}}^{-1}\left(\begin{array}{cc}
u & -v\\
v^{-1}u^2 & -u
\end{array}\right)T_{vu^{-1}}=\left(\begin{array}{cc}
0 & 0\\
v^{-1}u^2 & 0
\end{array}\right)$.
But Lemma \ref{conjug}(1) now shows that $T_{vu^{-1}}^{-1}M(a,b)T_{vu^{-1}}=M(a,b+avu^{-1})$. Since $k < l$ or $uu_b+vu_a \notin J(R)$, we have $u(b+avu^{-1})=uu_bx^k+vu_ax^l \in J(R)^k \setminus J(R)^{k+1}$, so we have $b+avu^{-1} \in J(R)^k \setminus J(R)^{k+1}$ as well, therefore 
\begin{multline*}
   \,\,\,\,\,\,\,\,\,\,\,\,\,\,  ((PT_{vu^{-1}})^{-1}N_1(PT_{vu^{-1}}))((PT_{vu^{-1}})^{-1}N_2(PT_{vu^{-1}}))\ldots ((PT_{vu^{-1}})^{-1}N_{2n-2}(PT_{vu^{-1}}))= \\ \left(M(a,b+avu^{-1})+T''\right)\left(\left(\begin{array}{cc}
0 & 0\\
v^{-1}u^2 & 0
\end{array}\right)+X''\right)
\end{multline*}
for some $T'' \in M_n(J(R)^{n-1})$, and some $X'' \in M_n(J(R))$.
But we have already proved in case (2) of this proof, that in this case there exist some $c',d' \in R$ such that $N_1N_2\ldots N_k \in \OO_{M(c',d')}$.

     \item
     $k=l$ and $uu_b+vu_a \in J(R)$: Then we consider the following cases.
     \begin{itemize}
         \item 
         $N'' \in M_2(J(R))$ or $N''=\left(\begin{array}{cc}
0 & 0\\
u & 0
\end{array}\right) + X''$, for some $u \in U(R)$ and some $X'' \in M_2(J(R))$. But then the statement follows by cases (1) and (2) of this proof.

\item
$N''=\left(\begin{array}{cc}
0 & w\\
0 & 0
\end{array}\right) + X''$, for some $X''\in M_2(J(R))$.  Then by conjugating equation (\ref{produkt}) with $T_{vu^{-1}}$, we get by Lemma \ref{conjug} that the product of our $2n-1$ nilpotents equals $(M(a,b+avu^{-1})+T'')N_1'N_1''$  for some $T''=\left(\begin{array}{cc}
t_1 & t_2\\
t_3 & t_4
\end{array}\right) \in M_2(J(R)^{n-1})$, $N_1'=\left(\begin{array}{cc}
0 & 0\\
v^{-1}u^2 & 0
\end{array}\right)+X_1'$ and $N_1''=\left(\begin{array}{cc}
0 & w\\
0 & 0
\end{array}\right) + X_1''$, for some $X_1',X_1''\in M_2(J(R))$.
The fact that $uu_b+vu_a \in J(R)$ implies that $b'=b+avu^{-1} \in J(R)^{k'} \setminus J(R)^{k'+1}$ for some $k' > l$. If $k' \leq n-2$, then Lemma \ref{deter} implies that $t_4=0$. Denote 
$X_1'=\left(\begin{array}{cc}
x_1 & x_2\\
x_3 & x_4
\end{array}\right)$ and $X_1''=\left(\begin{array}{cc}
y_1 & y_2\\
y_3 & y_4
\end{array}\right)$. Then, some direct calculations show that $(M(a,b')+T'')N_1'N_1''=M((ax_1+b'v^{-1}u^2+b'x_3)y_1+(ax_2+b'x_4)y_3, (ax_1+b'v^{-1}u^2+b'x_3+t_2v^{-1}u^2)w+(ax_1+b'v^{-1}u^2+b'x_3)y_2+(ax_2+b'x_4)y_4)$.

Finally, if $k > n-2$, then $b' \in J(R)^{n-1}$, so we can assume that $b'=0$. Again, some direct calculations yield $(M(a,0)+T'')N_1'N_1''=\left(\begin{array}{cc}
0 & ax_1w + t_2v^{-1}u^2w\\
0 & t_4v^{-1}u^2w
\end{array}\right)$. But now the statement holds by Lemma \ref{technical}(2).

\item 
$N''=\left(\begin{array}{cc}
u' & -v'\\
v'^{-1}u'^2 & -u'
\end{array}\right) + X''$, for some $u',v' \in U(R)$ and some $X''\in M_2(J(R))$. If $u'u_b+v'u_a \notin J(R)$, then the statement follows from the first part of case (3) of this proof. So, we can assume that $u'u_b+v'u_a \in J(R)$. Denote $uu_b+vu_a=j_1$ and $u'u_b+v'u_a=j_2$ for some $j_1, j_2 \in J(R)$. Then $v=-u_a^{-1}uu_b(1+j_1')$ and
$v'=-u_a^{-1}u'u_b(1+j_2')$, where $j_1'=-u^{-1}u_b^{-1}j_1, j_2'=-u'^{-1}u_b^{-1}j_2 \in J(R)$. But then we have
$vu'-uv', uu'^{-1}-vv'^{-1} \in J(R)$, so one can easily check that
$\left(\begin{array}{cc}
u & -v\\
v^{-1}u^2 & -u
\end{array}\right)\left(\begin{array}{cc}
u' & -v'\\
v'^{-1}u'^2 & -u'
\end{array}\right)
 \in M_2(J(R))$. However, this implies that $N'N'' \in M_2(J(R))$ as well.  Now, the statement follows by case (1) of this proof.
     \end{itemize}
      \end{itemize}
 
\item
$N'=\left(\begin{array}{cc}
0 & u\\
0 & 0
\end{array}\right) + X'$, for some $u \in U(R)$ and some $X' \in M_2(J(R))$. 
We consider the following cases:
\begin{itemize}
    \item
    $N'' \in M_2(J(R))$ or $N''=\left(\begin{array}{cc}
0 & 0\\
v & 0
\end{array}\right) + X''$ for some $v \in U(R)$ and some  $X'' \in M_2(J(R))$. The statement in this case follows by cases (1) and (2) of this proof.
 
    \item
    $N''=\left(\begin{array}{cc}
0 & v\\
0 & 0
\end{array}\right) + X''$, for some $v \in U(R)$ and some $X'' \in M_2(J(R))$. Thus, we have $N'N'' \in M_2(J(R))$, so the statement again follows by case (1) of this proof.

\item 
Finally, we have to consider the case $N''=\left(\begin{array}{cc}
v & -w\\
w^{-1}v^2 & -v
\end{array}\right) + X''$, for some $v,w \in U(R)$ and some $X''\in M_2(J(R))$. As we have already seen, if $k < l$ or if $vu_b+wu_a \notin J(R)$, then the statement follows from the first part of case (3) of this proof. So, we can assume that $k=l$ and that $vu_b+wu_a \in J(R)$. 
Then by conjugating equation (\ref{produkt}) with $T_{wv^{-1}}$, we get by Lemma \ref{conjug} that the product of our $2n-1$ nilpotents equals $(M(a,b+awv^{-1})+T'')N_1'N_1''$  for some $T''=\left(\begin{array}{cc}
t_1 & t_2\\
t_3 & t_4
\end{array}\right) \in M_2(J(R)^{n-1})$, $N_1'=\left(\begin{array}{cc}
0 & u\\
0 & 0
\end{array}\right)+X_1'$ and $N_1''=\left(\begin{array}{cc}
0 & 0\\
w^{-1}v^2 & 0
\end{array}\right) + X_1''$, for some $X_1'=\left(\begin{array}{cc}
x_1 & x_2\\
x_3 & x_4
\end{array}\right),X_1''=\left(\begin{array}{cc}
y_1 & y_2\\
y_3 & y_4
\end{array}\right)\in M_2(J(R))$.
The fact that $vu_b+wu_a \in J(R)$ implies that $b'=b+awv^{-1} \in J(R)^{k'} \setminus J(R)^{k'+1}$ for some $k' > k$.
Further calculations yield $(M(a,b')+T'')N_1'N_1''=\left(\begin{array}{cc}
m_1 & m_2
\\
m_3
&
0
\end{array}\right)$, where
$m_1=(a x_1 + b' x_3) y_1 + (a u + a x_2 + b' x_4) w^{-1}v^2 + t_1 u w^{-1}v^2 + (a u + a x_2 + b' x_4) y_3$, $m_2=(a x_1 + b' x_3) y_2 + (a u + a x_2 + b' x_4) y_4$ and $m_3=t_3 u w^{-1}v^2$. Observe that $m_1 \in J(R)^k \setminus J(R)^{k+1}$ and $m_2 \in J(R)^{k''} \setminus J(R)^{k''+1}$ for some $k'' > k$. Obviously, we also have $m_3 \in J(R)^{n-1}$. Now, the statement follows by Lemma \ref{technical}(1).
\end{itemize}
\end{enumerate}
So, we have proved that in every case, we have $N_1N_2 \ldots N_{2n-1} \in \OO_{M(a,b)}$ for some $a, b \in R$. By Lemma \ref{mab}, there exist $c,d \in R$ such that $N_1N_2 \ldots N_s \in \OO_{M(c,d)}$ and thus this theorem is finally proven.
\end{proof}

\begin{Example}
\label{counter}
      Note that Theorem \ref{product} does not hold if we take the product of fewer than $2n-1$ nilpotent matrices. Let $R$ be a finite local principal ring of cardinality $q^{n}$, where $q=p^r$ for some odd prime number $p$ and integer $r$ with $R/J(R) \simeq GF(q)$.  Denote $J(R)=(x)$ and suppose $n=2^k+1$ for some integer $k$. Since $2 \in U(R)$, there exists the element $(n-1)^{-1} \in R$. Observe that $\left(\left(\begin{array}{cc}
x & 1\\
(n-1)x & 0
\end{array}\right)\left(\begin{array}{cc}
0 & (n-1)^{-1}\\
x & 0
\end{array}\right)\right)^{n-1}=\left(\begin{array}{cc}
x & (n-1)^{-1}x\\
0 & x
\end{array}\right)^{n-1}=\left(\begin{array}{cc}
x^{n-1} & x^{n-1}\\
0 & x^{n-1}
\end{array}\right)$. So, $\left(\begin{array}{cc}
x^{n-1} & x^{n-1}\\
0 & x^{n-1}
\end{array}\right)$ is a product of $2n-2$ nilpotent matrices in $M_2(R)$. If $\left(\begin{array}{cc}
x^{n-1} & x^{n-1}\\
0 & x^{n-1}
\end{array}\right)$ is in $\OO_{M(a,b)}$ for some $a,b \in R$, then there exists an invertible matrix $P=\left(\begin{array}{cc}
\alpha & \beta\\
\gamma & \delta
\end{array}\right) \in M_2(R)$ such that $M(a,b)P=P\left(\begin{array}{cc}
x^{n-1} & x^{n-1}\\
0 & x^{n-1}
\end{array}\right)$. This gives us, among other equations, that $x^{n-1}\gamma=x^{n-1}(\gamma+\delta)=0$ and so $x^{n-1}\delta=0$. This implies that $\gamma, \delta \in J(R)$, but this is a contradiction with the assumption that $P$ is an invertible matrix. So, we have proved that 
$\left(\begin{array}{cc}
x^{n-1} & x^{n-1}\\
0 & x^{n-1}
\end{array}\right) \notin \OO_{M(a,b)}$ for any $a, b \in R$.
\end{Example}

\begin{Corollary}
\label{charact}
      Let $n \geq 2$ and let $R$ be a finite commutative local principal ring of cardinality $q^{n}$, where $q=p^r$ for some prime number $p$ and integer $r$ with $R/J(R) \simeq GF(q)$. Then $A \in M_2(R)$ is a product of $s \geq 2n-1$ nilpotent matrices if and only if there exist $a, b \in R$ such that $A \in \OO_{M(a,b)}$.
\end{Corollary}
\begin{proof}
   One side of the implication follows from Theorem \ref{product}. 
   
   So, choose $s \geq 2n-1$ and $A\in \OO_{M(a,b)}$. We shall prove that $A$ can be written as a product of $s$ nilpotents. It clearly suffices to show that $M(a,b)$ can be written as a product of $s$ nilpotents for any $a, b \in R$. We consider three cases.
   \begin{enumerate}
       \item 
       $a \in J(R)$: Then $M(a,b)=\left(\left(\begin{array}{cc}
0 & 1\\
0 & 0
\end{array}\right)\left(\begin{array}{cc}
0 & 0\\
1 & 0
\end{array}\right)\right)^t\left(\begin{array}{cc}
a & b\\
0 & 0
\end{array}\right)$ for any $t \geq 1$, so $M(a,b)$ can be written as a product of $2t+1$ nilpotents for any $t \geq 1$. But since $\left(\begin{array}{cc}
0 & 1\\
0 & 0
\end{array}\right)\left(\begin{array}{cc}
0 & 0\\
1 & 0
\end{array}\right)=\left(\begin{array}{cc}
0 & 1\\
0 & 0
\end{array}\right)\left(\begin{array}{cc}
-1 & 1\\
-1 & 1
\end{array}\right)\left(\begin{array}{cc}
0 & 0\\
1 & 0
\end{array}\right)$, we can see that $M(a,b)$ can also be written as a product of $2l$ of nilpotents for 
any $l \geq 2$.

\item 
$b \in J(R)$: By case (1), we know that $\left(\begin{array}{cc}
0 & 1\\
0 & 0
\end{array}\right)$ can be written as a product of $t \geq 3$ nilpotents for any $t \geq 3$. Since $M(a,b)=\left(\begin{array}{cc}
0 & 1\\
0 & 0
\end{array}\right)\left(\begin{array}{cc}
-1 & 1\\
-1 & 1
\end{array}\right)\left(\begin{array}{cc}
0 & 0\\
a & b
\end{array}\right)$ and 
$M(a,b)=\left(\begin{array}{cc}
0 & 1\\
0 & 0
\end{array}\right)\left(\begin{array}{cc}
0 & 0\\
1 & 0
\end{array}\right)\left(\begin{array}{cc}
-1 & 1\\
-1 & 1
\end{array}\right)\left(\begin{array}{cc}
0 & 0\\
a & b
\end{array}\right)$, we conclude that $M(a,b)$ can be written as a product of $t$ nilpotents for any $t \geq 3$ in this case as well.

\item 
$a, b \in U(R)$: Observe that $M(a,b)=\left(\begin{array}{cc}
0 & 1\\
0 & 0
\end{array}\right)\left(\begin{array}{cc}
0 & 0\\
a & 0
\end{array}\right)\left(\begin{array}{cc}
1 & a^{-1}b\\
-b^{-1}a & -1
\end{array}\right)$ and
$M(a,b)=\left(\begin{array}{cc}
0 & 1\\
0 & 0
\end{array}\right)\left(\begin{array}{cc}
0 & 0\\
-b & 0
\end{array}\right)\left(\begin{array}{cc}
0 & 1\\
0 & 0
\end{array}\right)\left(\begin{array}{cc}
1 & a^{-1}b\\
-b^{-1}a & -1
\end{array}\right)$, so the result follows.
   \end{enumerate}
\end{proof}

Let us state our final lemma.

\begin{Lemma}
\label{trace0}
    Let $F$ be a field. If $0 \neq A \in M_2(F)$ is a product of $2$ nilpotents, then $\Tr(A) \neq 0$.
\end{Lemma}
\begin{proof}
    Suppose $0 \neq A=N_1N_2$ is a product of two (nonzero) nilpotents. Then $N_1=u_1v_1^T$ and $N_1=u_2v_2^T$ for some $u_1,v_1,u_2,v_2 \in F^2$. Therefore $A=(v_1^Tu_2)u_1v_2^T$ and $\Tr(A)=(v_1^Tu_2)(v_2^Tu_1)$. Suppose now that $\Tr(A)=0$. Since $A \neq 0$, we have $v_1^Tu_2 \neq 0$, so we must have $v_2^Tu_1=0$. The fact that $N_2$ is nilpotent yields $v_2^Tu_2 =0$. 
    Because $v_2^Tu_1=v_2^Tu_2=0$ for some nonzero vector $v_2 \in F^2$, we have $u_1= \lambda u_2$ for some $0 \neq \lambda \in F$. However, $N_1$ is nilpotent, so $v_1^Tu_1=0$ and therefore also $v_1^Tu_2 = 0$, a contradiction.
\end{proof}

We can now state the main theorem.

\begin{Theorem}
\label{main}
    Let $R$ be a finite commutative local principal ring of cardinality $q^{n}$, where $q=p^r$ for some odd prime number $p$ and integer $r$ with $R/J(R) \simeq GF(q)$. 
    Choose $s \geq 2n-1$. Then the number of elements in $H(R)$ that can be written as a product of $s$ nilpotent elements is equal to
    \begin{equation*}
    \begin{cases} 
    q^2, \text { if } n=1 \text { and } s=1, \\
    q^3-q+1, \text { if } n=1 \text { and } s=2, \\
    q^{2n}-q^{n+1}+\frac{(q+2)q^{3n+1}}{q^2+q+1}+\frac{q^{3}+q^2+1}{q^2+q+1}-1, \text { otherwise}.
    \end{cases}
    \end{equation*}
\end{Theorem}
\begin{proof}
Denote $Z=\{x \in H(R); x$ can be written as a product of $ s$ nilpotents$\}$. Theorem \ref{izo} shows that $|Z|=|\{A \in M_2(R); A$ can be written as a product of $ s$ nilpotents$\}|$. 
If $n=1$ and $s=1$, then the statement follows from Lemma \ref{noofnil}.
If $n=1$ and $s=2$, then Theorem \ref{product} shows that if $A$ can be written as a product of $ s$ nilpotents, then $A \in \OO_{M(a,b)}$ for some $a, b \in R$. Now, if $a \neq 0$, then $M(a,b)=\left(\begin{array}{cc}
0 & 1\\
0 & 0
\end{array}\right)\left(\begin{array}{cc}
-b & -a^{-1}b^2\\
a & b
\end{array}\right)$ is a product of two nilpotents. However, if $a=0$ and $b \neq 0$, then $M(a,b)$ is not a product of two nilpotents by Lemma \ref{trace0}. Lemma 3.3 from \cite{dolzanid} shows that for $a \neq 0$ we have $O_{M(a,b)}=O_{M(a',b')}$ for some $a',b' \in R$ if and only if $a'=a$.
Obviously $M(0,0)$ can also be written as a product of two nilpotents, so $|Z|=\sum\limits_{a \neq 0}|O_{M(a,0)}|+1.$ From \cite[Lemma 3.5]{dolzanid}, we have
$|O_{M(a,0)}|=q(q+1)$ for any $a \neq 0$, so $|Z|=(q-1)q(q+1)+1=q^3-q+1$.

From here onwards, we can therefore assume that $s \geq 3$.
It follows from Corollary \ref{charact} that $|Z|=|\{A \in M_2(R); $ there exist $a,b \in R$ such that $A \in M(a,b)\}|$.
Finally, \cite[Lemma 3.3 and Theorem 3.8]{dolzanid} give us $$|Z|=q^{2n}-q^{n+1}+\frac{(q+2)q^{3n+1}}{q^2+q+1}+\frac{q^{3}+q^2+1}{q^2+q+1}-1.$$
\end{proof}

\begin{Example}
\label{exz9}
   Examine the ring $R=\ZZ_9$. Lemma \ref{noofnil} tells us that there are $729$ nilpotents in $H(R)$. By Theorem \ref{main}, there exist exactly $897$ out of $2673$ noninvertible elements in $H(R)$ that can be written as a product of $3$ or more nilpotents. 
\end{Example}

\bigskip

{\bf Statements and Declarations} \\

The author states that there are no competing interests. 

\bigskip

\bibliographystyle{amsplain}
\bibliography{biblio}

\bigskip

\end{document}